\documentclass[oneside,english]{amsart}
\usepackage{amstext}
\usepackage{amsthm}
\usepackage{amssymb}

\makeatletter

\numberwithin{equation}{section}
\numberwithin{figure}{section}
\theoremstyle{definition}
\newtheorem*{defn*}{Definition}
\theoremstyle{plain}
\newtheorem{thm}{Theorem}[section]
\theoremstyle{remark}
\newtheorem{rem}[thm]{Remark}
\theoremstyle{plain}
\newtheorem{lem}[thm]{Lemma}
\theoremstyle{plain}
\newtheorem{prop}[thm]{Proposition}
\theoremstyle{plain}
\newtheorem{cor}[thm]{Corollary}

\makeatother

\begin{document}
\title{Quantitative Magnetic Isoperimetric Inequality}
\author{Rohan Ghanta, Lukas Junge and Léo Morin }
\date{May 9, 2023}
\address{\hspace{-3.25ex}Institut for Matematik, Aarhus Universitet, Ny Munkegade
8000 Aarhus C, Danmark}
\address{\hspace{-3.25ex}ghanta@alumni.princeton.edu $\ \ $junge@math.au.dk$\ \ $leo.morin@ens-rennes.fr}
\begin{abstract}
In 1996 Erdős showed that among planar domains of fixed area, the
smallest principal eigenvalue of the Dirichlet Laplacian with a constant
magnetic field is uniquely achieved on the disk. We establish a quantitative
version of this inequality, with an explicit remainder term depending
on the field strength that measures how much the domain deviates from
the disk. 
\end{abstract}

\maketitle

\section{Introduction}

\noindent To solve a problem in Probability and Mathematical Physics
\cite{ErdProb1},\cite{ErdProb2}, Erdős developed the magnetic isoperimetric
inequality \cite{key-79}. It generalizes the Faber-Krahn inequality
to the magnetic Laplacian. Starting with Pólya and Szegő \cite{key-80},
Faber-Krahn-type results have been established by proving rearrangement
inequalities. The inclusion of a magnetic field, however, makes it
notoriously difficult to implement the standard symmetrization methods.
Erdős met the challenge head on: he managed to prove a magnetic rearrangement
inequality, which is reminiscent of the celebrated Pólya-Szegő inequality
but with an interesting caveat. Such symmetry results with a magnetic
field are--alas!--very few and far between \cite{1},\cite{key-82}. 

Still another compelling feature is that rearrangements alone are
not sufficient for arguing the magnetic isoperimetric inequality.
This stands in sharp contrast to the classical Faber-Krahn setting.
To complete the proof Erdős introduced a new inequality, tailored
specifically for a magnetic Schrödinger operator on a disk and for
which there exists no analog in the absence of a magnetic field. 

We improve Erdős' result. He showed that if a planar domain is not
a disk, then the principal eigenvalue of the Dirichlet magnetic Laplacian
is strictly larger on that domain than on the disk of same area. We
take the next step and establish stability: if the principal eigenvalue
of the magnetic Laplacian is just slightly larger on a planar domain
than on the disk of same area, then that domain is only slightly different
from the disk. Faint perturbations of the smallest principal eigenvalue
will not induce a dramatic change in the underlying geometry--and
this dynamic is very sensitive to the field strength. We prove our
stability estimate with a remainder term that quantifies the difference
between the domain and the disk. 

Quantitative Faber-Krahn-type inequalities have been developed almost
exclusively around the classical theory of rearrangements. Fueled
in large part by the seminal work of Fusco et al. \cite{key-2}, the
last decade has given rise to an entire industry now devoted to the
stability of a remarkable range of geometric and functional inequalities.
Our paper provides the first stability result with a magnetic field.
And here, the well-established rearrangement framework is no longer
sufficient. 

\section{Statement of Problem and Main Result}

\noindent Let $\Omega\subset\mathbb{R}^{2}$ be a bounded, connected
open set with a smooth boundary. The principal eigenvalue of the Dirichlet
magnetic Laplacian on the planar domain $\Omega$ is
\begin{equation}
\lambda(B,\Omega):=\inf_{f\in H_{0}^{1}(\Omega)}\frac{\int_{\Omega}\vert(-i\nabla-\alpha)f\vert^{2}dx}{\int_{\Omega}\vert f\vert^{2}dx},\label{eq.lambda}
\end{equation}
where $\alpha=\frac{B}{2}\left(-x_{2},x_{1}\right)$ is a magnetic
vector potential generating a homogeneous magnetic field of strength
$B\geq0$, i.e. $\text{rot(}\alpha)=B$. We denote by $D_{R}$ a disk
of radius $R$, centered at the origin, with the same area as $\Omega$,
i.e. $\left|\Omega\right|=\left|D_{R}\right|=\pi R^{2}$. 

In 1996 Erdős \cite{key-79} proved the magnetic isoperimetric inequality
\begin{equation}
\lambda(B,\Omega)\geq\lambda(B,D_{R}),\label{MFK}
\end{equation}
with equality if and only if $\Omega$ is a disk. In the absence of
a magnetic field, i.e. $B=0$, his result reduces to the usual Faber-Krahn
inequality. 

In this paper, we want to add to the right-hand side of (\ref{MFK})
a remainder term that measures how much the planar domain $\Omega$
deviates from being a disk. This would make it possible to understand
the shape of $\Omega$ now in terms of how close it is to achieving
equality in (\ref{MFK}). Cf. \cite{key-3} \& references therein.

We measure the difference between $\Omega$ and the disk in the usual
way in terms of the interior deficiency and the Fraenkel asymmetry
of the domain. 
\begin{defn*}
The interior deficiency (asymmetry) of a set is defined as
\[
\mathcal{A}_{I}\left(\Omega\right):=\frac{R-\rho_{-}(\Omega)}{R},
\]
where $\rho_{-}(\Omega)$ denotes the radius of the largest ball contained
in $\Omega$, and $R$ as above is the radius of $D_{R}$. 
\end{defn*}
\begin{defn*}
The Fraenkel asymmetry of a set is defined as
\[
\mathcal{A}_{F}\left(\Omega\right):=\inf_{x_{0}\in\mathbb{R}^{2}}\frac{\vert\Omega\Delta(x_{0}+D_{R})\vert}{2\vert\Omega\vert}.
\]
\end{defn*}
\noindent Both asymmetries are bounded by one and vanish if and only
if the set is a disk. 

Our main result is a quantitative version of the magnetic isoperimetric
inequality. 
\begin{thm}
\label{thm:Main Result}Let $\mathcal{A}\left(\Omega\right)$ denote
either the interior asymmetry or the Fraenkel asymmetry. In the case
of the interior asymmetry we also assume $\Omega$ is simply connected.
Then there is a universal constant $c>0$, independent of $\Omega$
and $B$, such that
\begin{equation}
\lambda(B,\Omega)\geq\lambda(B,D_{R})(1+ce^{-\frac{5}{6}BR^{2}}\mathcal{A}(\Omega)^{\frac{10}{3}})\,.\label{eq:quantitative}
\end{equation}
Moreover, if $0\leq BR^{2}\leq\frac{1}{\pi}$, then
\begin{equation}
\lambda(B,\Omega)\geq\lambda(B,D_{R})(1+c\mathcal{A}(\Omega)^{3})\,.\label{eq.B.small}
\end{equation}
\end{thm}

\begin{rem}
\noindent \label{rem:scaling}The quantity $\mathcal{A}\left(\Omega\right)$
is scale invariant. Furthermore $\lambda$ scales like $t^{2}\lambda(B,t\Omega)=\lambda(t^{2}B,\Omega)$
for $t>0$, so the factor $BR^{2}$ appearing in our constant is the
natural parameter for this problem. 
\end{rem}

In the absence of a magnetic field, i.e. $B=0$, the estimate in (\ref{eq.B.small})
reduces to Hansen and Nadirashvili's quantitative Faber-Krahn inequality
with the asymmetry cubed \cite{key-4},\cite{Bhattacharya}. More
recently, Brasco et al. \cite{key-6} proved it with the square power:
this is the sharp form, because the exponent cannot be any smaller
\cite{key-7},\cite{key-8}. Our magnetic version in (\ref{eq:quantitative})
should likewise instead have the square of the asymmetry and, in principle,
one could adapt Brasco et al.'s argument to achieve this. Their state-of-the-art
methods, however, are nonconstructive and will not yield an explicit
constant. This would make it impossible to understand the pertinent
role of the magnetic field strength $B$ in the stability of Erdős'
inequality. 

Our methods, on the other hand, yield an explicit constant with a
natural dependence on the field strength. Physical intuition suggests
that as $B\rightarrow\infty$ the principal eigenfunctions start to
localize on a length scale proportional to $1/\sqrt{B}$, away from
the boundary, and therefore $\lambda(B,\cdot)$ becomes less sensitive
to the shape of the domain: it can but faintly distinguish between
even very dissimilar shapes, and the little sensitivity that remains
comes from the fact that these eigenfunctions can still feel about
near the boundary with their exponentially small tails. Now $\Omega$
can look rather different from $D_{R}$ and yet $\lambda(B,\Omega)\approx\lambda(B,D_{R})$:
a strong magnetic field compromises stability. We manage to capture
this picture in (\ref{eq:quantitative}) with our constant which vanishes,
exponentially, as $B\rightarrow\infty$. 

To prove his Faber-Krahn-type inequality in (\ref{MFK}), Erdős started
out in the usual way by establishing a rearrangement inequality. See
Lemma \ref{prop:Erdos}. While there are certainly nontrivial magnetic
aspects to the argument, Erdős essentially mimicked the standard proof
\cite{Talenti} of the analogous Pólya-Szegő inequality using the
coarea formula and the isoperimetric inequality. But in imposing the
Pólya-Szegő scheme on his problem, he was forced to change the magnetic
field on the disk. The vector potential on the right-hand side of
(\ref{eq:1}) is no longer the same: and thus his magnetic rearrangement
inequality cannot readily imply (\ref{MFK}) in the same way that
the Pólya-Szegő inequality yields Faber-Krahn.

To deal with this mis-match between the magnetic fields on $\Omega$
and $D_{R}$, Erdős developed the \textit{comparison lemma} on the
disk. See Remark \ref{rem:Remark compare}. It compares the ground-state
energies of the operator on the right-hand side of (\ref{eq:1}) corresponding
to different magnetic fields. This in turn allowed him to recover
the original magnetic field on $D_{R}$ and finish proving (\ref{MFK}).
His comparison lemma is built on the variational principle and has
nothing to do with rearrangements. And unlike his rearrangement inequality,
it has no analog in the absence of a magnetic field. 

To prove our stability estimate in Theorem \ref{thm:Main Result},
we also start out in the usual way by establishing a quantitative
version of Erdős' rearrangement inequality. See Proposition \ref{prop:rearrangement}.
This is nothing new: in the absence of a magnetic field, i.e. $B=0$,
it just reduces to the quantitative version of the Pólya-Szegő inequality
that was used in proving stability of Faber-Krahn \cite{key-3}. Here
we mimic Erdős' proof but instead apply the \textit{quantitative isoperimetric
inequality} on the level sets. 
\begin{thm}
\label{thm:quant iso}Let $U\subset\mathbb{R}^{2}$ be a bounded set
with smooth boundary, and let $\mathcal{P}(U)$ denote the perimeter
of $U$. Let $\mathcal{A}(U)$ denote either the interior asymmetry
or the Fraenkel asymmetry. In the case of the interior asymmetry we
also assume $U$ is simply connected. Then there is a universal constant
$c>0$ such that 
\[
\mathcal{P}(U)\geq2\sqrt{\pi}\left|U\right|^{\frac{1}{2}}\left(1+c\mathcal{A}(U)^{2}\right).
\]
\end{thm}

\noindent This was first proved by Bonnesen in 1924 for simply connected
planar sets using the interior asymmetry \cite{key-5},\cite{key-8}.
In 2008 Fusco et al. proved a more general version using the Fraenkel
asymmetry \cite{key-2}. Theorem \ref{thm:quant iso} forms the backbone
of the first part of the paper. 

In Lemma \ref{lem:quantitative comparision} we establish a quantitative
version of Erdős' comparison lemma. Now this is really a new estimate,
which stands completely outside of the rearrangement framework--and
it only enters the scene when $B$ is large. 

In Corollary \ref{cor:goal of part 1} we present two very different
lower bounds on the quantity $\lambda(B,\Omega)-\lambda(B,D_{R})$,
both involving the asymmetry of the level sets of the principal eigenfunction
corresponding to $\lambda(B,\Omega)$. The first bound, (\ref{eq:rearrange error}),
is based on our quantitative version of the rearrangement inequality.
The second bound, (\ref{eq: compare error}), is based on our quantitative
version of the comparison lemma. 

As usual, the main difficulty lies in going from the asymmetry of
these level sets in Corollary \ref{cor:goal of part 1} to the asymmetry
of the whole domain. We deal with this in the second part of the paper.
When $B$ is small, we operate entirely within the rearrangement framework
just as in the classical Faber-Krahn setting. Here our argument is
a direct perturbation of Hansen and Nadirashvili's proof of their
quantitative Faber-Krahn inequality \cite{key-4}. We only use the
first bound, given in (\ref{eq:rearrange error}), of Corollary \ref{cor:goal of part 1}
which is based on the quantitative version of the rearrangement inequality.
This is enough to prove the estimate in (\ref{eq.B.small}) of Theorem
\ref{thm:Main Result}. 

But as $B$ increases, our weak-field adaptation of Hansen and Nadirashvili's
technique breaks down: with a strong magnetic field, the rearrangement
framework alone is no longer sufficient for establishing stability.
Here we make full use of both the quantitative version of the rearrangement
inequality \textit{and now} our quantitative version of the comparison
lemma. A distinctive feature of our argument is the necessary interplay
between the traditional bound in (\ref{eq:rearrange error})--rooted
firmly within the paradigmatic framework of rearrangement inequalities--and
our \textit{magnetic bound} in (\ref{eq: compare error}), which is
unique to our problem and irreducible to any other estimate used in
establishing stability of a Faber-Krahn-type inequality. 

\part{The Magnetic Isoperimetric Inequality}

\noindent Here we re-prove Erdős' magnetic isoperimetric inequality
but with a remainder term involving the asymmetry of the level sets
of the principal eigenfunction corresponding to $\lambda\left(B,\Omega\right)$.
This is given as Corollary \ref{cor:goal of part 1}. The quantitative
isoperimetric inequality plays an essential role. 

\section{The Magnetic Rearrangement Inequality}

\noindent Standard elliptic theory tells us that the principal eigenfunction
corresponding to $\lambda(B,\Omega)$ is a complex-valued analytic
function. The first ingredient in Erdős' proof is a rearrangement
inequality. He proved the following. 
\begin{lem}
\noindent \label{prop:Erdos}Let $f$, $\|f\|_{2}=1$ be a complex-valued
analytic function on $\Omega$ that vanishes on the boundary, and
let $\left|f\right|^{*}$ denote the symmetric decreasing rearrangement
of $\left|f\right|$. Then there exists a vector potential $\tilde{\alpha}(x)=\frac{a\left(\left|x\right|\right)}{\left|x\right|}\left(-x_{2},x_{1}\right)$,
where $a\left(\left|x\right|\right)$ is a function satisfying $0\leq a\left(\left|x\right|\right)\leq\frac{B\left|x\right|}{2}$,
such that 
\begin{equation}
\int_{\Omega}\left|\left(-i\nabla-\alpha\right)f\right|^{2}dx\geq\int_{D_{R}}\left|\left(-i\nabla-\tilde{\alpha}\right)\left|f\right|^{*}\right|^{2}dx+B-\int_{D_{R}}\text{rot}\left(\tilde{\alpha}\right)\left|f\right|^{*2}dx.\label{eq:1}
\end{equation}
\end{lem}

\noindent This is analogous to the celebrated Pólya-Szegő inequality
but with some caveats: 
\begin{enumerate}
\item The magnetic field on the disk is no longer the same. Our vector potential
$\alpha=\frac{B}{2}\left(-x_{2},x_{1}\right)$ corresponds to a homogeneous
field of strength $B$. Now $\tilde{\alpha}$ corresponds to a radially
symmetric but \textit{inhomogeneous} field.
\item The potential $\tilde{\alpha}$ depends on $f$, because Erdős constructed
$a\left(\left|x\right|\right)$ from the level sets of $\left|f\right|$; 
\item in particular, if $a\left(\left|x\right|\right)=\frac{B\left|x\right|}{2}$,
then the level set $\left\{ \left|f\right|>\left|f\right|^{*}(x)\right\} $
is a disk. 
\end{enumerate}
Lemma \ref{prop:Erdos} yields a lower bound on $\lambda(B,\Omega)$.
Had the vector potential remained unchanged, (\ref{eq:1}) would have
readily implied $\lambda(B,\Omega)\geq\lambda(B,D_{R})$. 

In this section we prove a quantitative version of his rearrangement
inequality, and we write the right-hand side more conveniently in
terms of polar coordinates. 
\begin{prop}
\noindent \label{prop:rearrangement} Let $f$, $\|f\|_{2}=1$ be
as in the statement of Lemma \ref{prop:Erdos}, and $q\left(\left|x\right|\right):=\left|f\right|^{*}(x)$.
Then there exists a bounded function $a\left(\left|x\right|\right)$,
depending on $f$ and $B$, such that\footnote{We use here the following convention for the interior asymmetry. If
the open set $U$ is not simply connected, we define $\mathcal{A}_{I}(U)$
to be the asymmetry of the smallest simply connected set containing
$U$. Since $\Omega$ is simply connected, this will not change the
final value of $\mathcal{A}_{I}(\Omega)$. This convention allows
us to use Theorem \ref{thm:quant iso} for the level sets of $\left|f\right|$.}

\noindent 
\[
\int_{\Omega}\left|\left(-i\nabla-\alpha\right)f\right|^{2}dx\geq B+2\pi\int_{0}^{R}\left(q'(r)+a(r)q(r)\right)^{2}\left(1+c\mathcal{A}^{2}\left(\left\{ \left|f\right|>q(r)\right\} \right)\right)^{2}rdr,
\]

\noindent and
\begin{equation}
0\leq a(r)\leq\frac{Br}{2}\left(1+c\mathcal{A}^{2}\left(\left\{ \left|f\right|>q(r)\right\} \right)\right)^{-2}\leq\frac{Br}{2}\,,\label{eq:new  potential}
\end{equation}
where $c>0$ is a universal constant independent of $B$ and $\Omega$. 
\end{prop}

In the absence of the asymmetry term, the expression on the right-hand
side indeed coincides with that of (\ref{eq:1}). See Proof of Lemma
\ref{lemma.disk} in the appendix.

\subsection{The Proof of Proposition \ref{prop:rearrangement}}

Erdős proved his rearrangement inequality within the standard Pólya-Szegő
scheme \cite{Talenti} using the coarea formula and the isoperimetric
inequality, which we replace with its quantitative version. 

To use the coarea formula, first we need a real-valued function. By
modifying the magnetic vector potential, we can work with $\left|f\right|$
instead. 
\begin{lem}
\label{lem:real} Let $f$ be as in the statement of Lemma \ref{prop:Erdos},
and $\Omega_{0}:=\Omega\setminus\lbrace f=0\rbrace$. Let $\theta:\Omega_{0}\mapsto\left[0,2\pi\right)$
be such that $f=\left|f\right|e^{i\theta}$. Since $\Omega_{0}$ has
full measure, $w:=\alpha-\nabla\theta$ is defined almost everywhere
and $\text{rot}\left(w\right)=B$.  Then, with $w^{\perp}:=\left(-w_{2},w_{1}\right)$,
\[
\int_{\Omega}\left|\left(-i\nabla-\alpha\right)f\right|^{2}dx=B+\int_{\Omega}\left|\nabla\left|f\right|+w^{\perp}\left|f\right|\right|^{2}dx.
\]
\end{lem}

\begin{proof}
Since $\left|f\right|\in H_{0}^{1}\left(\Omega\right)$ and $w$ is
real-valued, 
\[
\int_{\Omega}\left|\left(-i\nabla-\alpha\right)f\right|^{2}dx=\int_{\Omega}\left|\left(-i\nabla-w\right)\left|f\right|\right|^{2}dx=\int_{\Omega}\left(\left|\nabla\left|f\right|\right|^{2}+\left|w^{\perp}\right|^{2}\left|f\right|^{2}\right)dx.
\]
Note $w$ is smooth a.e. By completing the square and integrating
by parts, 
\begin{align*}
\int_{\Omega}\left|\left(-i\nabla-\alpha\right)f\right|^{2}dx & =\int_{\Omega}\left(\left|\nabla\left|f\right|+w^{\perp}\left|f\right|\right|^{2}-2\left|f\right|w^{\perp}\cdot\nabla\left|f\right|\right)dx\\
 & =\int_{\Omega}\left(\left|\nabla\left|f\right|+w^{\perp}\left|f\right|\right|^{2}+\left|f\right|^{2}\text{div}(w^{\perp})\right)dx.
\end{align*}
Since $\text{rot}\left(w\right)=B$, the lemma follows. 
\end{proof}
Then we use the coarea formula and arrive at an expression involving
an integral over the level sets of $\left|f\right|$. 
\begin{lem}
\label{lem: coarea} Let $f,w^{\perp}$ be as in the statement of
Lemma \ref{lem:real}. Then,  
\begin{equation}
\int_{\Omega}\left|\nabla\left|f\right|+w^{\perp}\left|f\right|\right|^{2}dx\geq\int_{0}^{\infty}dz\ (1-B\Phi(z)z)^{2}\int_{\lbrace\left|f\right|=z\rbrace}\vert\nabla\left|f\right|\vert,\label{eq:coarea}
\end{equation}
with
\begin{equation}
\Phi(z):=\frac{\left|\left\{ \left|f\right|>z\right\} \right|}{\int_{\left\{ \left|f\right|=z\right\} }\left|\nabla\left|f\right|\right|}.\label{eq: Phi}
\end{equation}
\end{lem}

\noindent If there is no magnetic field, i.e. $B=0$, and $f$ is
a positive function, then the relation in (\ref{eq:coarea}) reduces
to the usual coarea formula used in the proof of the Pólya-Szegő inequality
\cite{Talenti}.
\begin{proof}[Proof of Lemma \ref{lem: coarea}]
 There exists $w'$ orthogonal to $\nabla\left|f\right|$ and $\varphi:\Omega\mapsto\mathbb{R}$
such that $w^{\perp}=-\varphi\nabla\left|f\right|+w'$. By the Pythagorean
theorem, 
\begin{align*}
\int_{\Omega}\left|\nabla\left|f\right|+w^{\perp}\left|f\right|\right|^{2}dx & =\int_{\Omega}\left(\left|\left(1-\varphi\left|f\right|\right)\nabla\left|f\right|\right|^{2}+\left|w'\left|f\right|\right|^{2}\right)dx\\
 & \geq\int_{\Omega}\left|\left(1-\varphi\left|f\right|\right)\nabla\left|f\right|\right|^{2}dx.
\end{align*}
Now we are in a position to use the coarea formula:
\begin{align*}
\int_{\Omega}\left|\left(1-\varphi\left|f\right|\right)\nabla\left|f\right|\right|^{2}dx & =\int_{0}^{\infty}dz\int_{\left\{ \left|f\right|=z\right\} }\left(1-\varphi z\right)^{2}\left|\nabla\left|f\right|\right|\\
 & \geq\int_{0}^{\infty}dz\,\frac{\left(\int_{\{\left|f\right|=z\rbrace}(1-\varphi z)\vert\nabla\left|f\right|\vert\right)^{2}}{\int_{\lbrace\left|f\right|=z\rbrace}\vert\nabla\left|f\right|\vert}.
\end{align*}
We use Stokes' theorem on the level sets. For almost all $z>0$, the
level set $\lbrace\left|f\right|=z\rbrace$ is regular by Sard's theorem.
Thus
\[
B\vert\lbrace\left|f\right|>z\rbrace\vert=\int_{\lbrace\left|f\right|>z\rbrace}\text{rot}\left(w\right)=\int_{\lbrace\left|f\right|=z\rbrace}w\cdot\tau,
\]
where $\tau=\frac{(\nabla\left|f\right|)^{\perp}}{\left|(\nabla\left|f\right|)^{\perp}\right|}$.
Since $w\cdot\tau=\varphi\left|\nabla\left|f\right|\right|$, we conclude
\[
\int_{\Omega}\left|\nabla\left|f\right|+w^{\perp}\left|f\right|\right|^{2}dx\geq\int_{0}^{\infty}dz\,\frac{\left(\int_{\lbrace\left|f\right|=z\rbrace}\vert\nabla\left|f\right|\vert-Bz\vert\lbrace\left|f\right|>z\rbrace\vert\right)^{2}}{\int_{\lbrace\left|f\right|=z\rbrace}\vert\nabla\left|f\right|\vert}.
\]
The lemma follows from the definition of $\Phi$ in (\ref{eq: Phi}). 
\end{proof}
With the coarea-type estimate in (\ref{eq:coarea}), Erdős applied
the isoperimetric inequality on the level sets of $\left|f\right|$
to prove his rearrangement inequality; and when $B=0$, his argument
reduces to the standard proof of the Pólya-Szegő inequality \cite{Talenti}.
Below we instead apply the quantitative isoperimetric inequality on
these level sets. 
\begin{proof}[Proof of Proposition \ref{prop:rearrangement}]
 From Lemma \ref{lem:real}, Lemma \ref{lem: coarea} and Hölder's
inequality
\[
\int_{\Omega}\left|\left(-i\nabla-\alpha\right)f\right|^{2}dx\geq B+\int_{0}^{\infty}dz\ (1-B\Phi(z)z)^{2}\frac{\left|\left\{ \left|f\right|=z\right\} \right|^{2}}{\int_{\left\{ \left|f\right|=z\right\} }\left|\nabla\left|f\right|\right|^{-1}}\,.
\]
By Sard's theorem, the denominator is non-vanishing for almost all
$z>0$. And since $q$ is the rearrangement of $\left|f\right|$,
\begin{equation}
q(r)=F^{-1}\left(\pi r^{2}\right)\ \text{where}\ F(z):=\left|\left\{ \left|f\right|>z\right\} \right|.\label{eq: Definition of rearrangement}
\end{equation}
By the coarea formula, again for almost all $z>0$ 
\begin{equation}
F(z)=\int_{z}^{\infty}d\xi\int_{\left\{ \left|f\right|=\xi\right\} }\left|\nabla\left|f\right|\right|^{-1}\ \text{and}\ F^{\prime}(z)=-\int_{\left\{ \left|f\right|=z\right\} }\left|\nabla\left|f\right|\right|^{-1}.\label{eq: F prime}
\end{equation}
Then,
\begin{equation}
\int_{\Omega}\left|\left(-i\nabla-\alpha\right)f\right|^{2}dx\geq B-\int_{0}^{\infty}\left(1-B\Phi(z)z\right)^{2}\left|\left\{ \left|f\right|=z\right\} \right|^{2}F^{\prime}(z)^{-1}\,dz.\label{eq:index}
\end{equation}
Now we do a change of variable $z=q(r)$ and apply the isoperimetric
inequality, Theorem \ref{thm:quant iso}, on the level sets: $\left|\left\{ \left|f\right|=q(r)\right\} \right|\geq2\pi r\left(1+c\mathcal{A}^{2}\left(\left\{ \left|f\right|>q(r)\right\} \right)\right)$.
We write $\mathcal{A}^{2}$ for short. Then, 
\[
\int_{\Omega}\left|\left(-i\nabla-\alpha\right)f\right|^{2}dx\geq B+\int_{0}^{R}\left(1-B\Phi\left(q(r)\right)q(r)\right)^{2}\frac{\left(2\pi r\right)^{2}q^{\prime}(r)}{F'(q(r))}\left(1+c\mathcal{A}^{2}\right)^{2}dr.
\]
Since $q'(r)=2\pi rF'(q(r))^{-1}$, 
\[
\int_{\Omega}\left|\left(-i\nabla-\alpha\right)f\right|^{2}dx\geq B+2\pi\int_{0}^{R}\left[q'(r)-\frac{2\pi rB\Phi(q(r))}{F'(q(r))}q(r)\right]^{2}(1+c\mathcal{A}^{2})^{2}rdr.
\]
Writing $a(r):=-2\pi rBF'(q(r))^{-1}\Phi(q(r))$, we deduce our rearrangement
inequality. 

It remains to prove the upper bound in (\ref{eq:new  potential}).
By Hölder's inequality
\[
-F'(q(r))=\int_{\lbrace\left|f\right|=q(r)\rbrace}\left|\nabla\left|f\right|\right|^{-1}\geq\vert\left\{ \left|f\right|=q(r)\right\} \vert^{2}\left(\int_{\lbrace\left|f\right|=q(r)\rbrace}\vert\nabla\left|f\right|\vert\right)^{-1},
\]
and by the isoperimetric inequality, Theorem \ref{thm:quant iso},
\[
a(r)\leq2\pi rB\frac{\left|\left\{ \left|f\right|>q(r)\right\} \right|}{\left|\left\{ \left|f\right|=q(r)\right\} \right|^{2}}\leq\frac{Br}{2}\left(1+c\mathcal{A}^{2}\left(\left\{ \left|f\right|>q(r)\right\} \right)\right)^{-2}.
\]
This concludes the proof of Proposition \ref{prop:rearrangement}. 
\end{proof}

\section{\label{sec:The-Comparison-Lemma}The Comparison Lemma }

\noindent The second ingredient in Erdős' proof is a comparison lemma,
which makes it possible to recover from the right-hand side of (\ref{eq:1})
the original potential $\alpha$ on the disk. In this section we prove
a quantitative version of his comparison lemma. 

For a potential $\tilde{\alpha}=\frac{a\left(\left|x\right|\right)}{\left|x\right|}\left(-x_{2},x_{1}\right)$,
with $a\in L^{\infty}\left(\left(0,R\right)\right)$, we consider
the ground-state energy of the operator $\left(-i\nabla-\tilde{\alpha}\right)^{2}-\text{rot}\left(\tilde{\alpha}\right)$
restricted to radial functions on the disk, again written more conveniently
in terms of polar coordinates
\begin{equation}
\mathfrak{e}\left(a(r)\right):=\inf_{q\in H_{0}^{1,\text{rad}}\left(D_{R}\right)}\frac{2\pi\int_{0}^{R}\left(q^{\prime}(r)+a(r)q(r)\right)^{2}rdr}{2\pi\int_{0}^{R}q(r)^{2}rdr},\label{eq:disk energy}
\end{equation}
where $H_{0}^{1,\text{rad}}(D_{R}):=\left\{ q:[0,R]\rightarrow\mathbb{R}\ \text{such that}\ x\mapsto q(\vert x\vert)\ \text{belongs to}\ H_{0}^{1}(D_{R})\right\} .$ 

The function $a(r)=\frac{Br}{2}$ corresponds to the original potential
$\alpha=\frac{B}{2}\left(-x_{2},x_{1}\right)$, and since $\text{rot}\left(\alpha\right)=B$,
\begin{equation}
B+\mathfrak{e}\left(Br/2\right)=\inf_{q\in H_{0}^{1,\text{rad}}\left(D_{R}\right)}\frac{\int_{D_{R}}\left|\left(-i\nabla-\alpha\right)q(\left|x\right|)\right|^{2}dx}{\int_{D_{R}}q\left(\left|x\right|\right)^{2}dx}\geq\lambda(B,D_{R}).\label{eq:variational}
\end{equation}
We compare the ground-state energies for different potentials on the
disk. 
\begin{lem}
\noindent \label{lem:quantitative comparision}Let $q_{a}$ be a normalized
minimizer for the energy $\mathfrak{e}\left(a(r)\right)$ in (\ref{eq:disk energy}).
Let 
\begin{equation}
u_{a}(r):=\exp\left(-2\int_{0}^{r}a(s)\,ds\right)\ \ \text{and\ }\ p_{a}(r):=q_{a}(r)u_{a}(r)^{-\frac{1}{2}}.\label{eq:definition of u}
\end{equation}
Then for $a,\tilde{a}\in L^{\infty}\left(\left(0,R\right)\right)$,
\begin{equation}
\mathfrak{e}\left(a(r)\right)\geq\mathfrak{e}\left(\tilde{a}(r)\right)+\frac{2\int_{0}^{R}\left(\tilde{a}-a\right)p_{a}\left|p_{a}^{\prime}\right|u_{\tilde{a}}rdr}{\int_{0}^{R}p_{a}^{2}u_{\tilde{a}}rdr}.\label{eq: comparision error}
\end{equation}
\end{lem}

\begin{rem}
\noindent \label{rem:Remark compare}Our bound in (\ref{eq: comparision error})
implies Erdős'\textbf{ comparison lemma}: if $a\leq\tilde{a}$, then
$\mathfrak{e}\left(a(r)\right)\geq\mathfrak{e}\left(\tilde{a}\left(r\right)\right)$.
See Lemma 3.1 in \cite{key-79}. 
\end{rem}

\begin{proof}
We write
\begin{equation}
\mathfrak{e}\left(a(r)\right)=\inf_{p\in H_{0}^{1,\text{rad}}\left(D_{R}\right)}\frac{\int_{0}^{R}\left(p^{\prime}\right)^{2}u_{a}rdr}{\int_{0}^{R}p^{2}u_{a}rdr}=\frac{\int_{0}^{R}\left(p_{a}^{\prime}\right)^{2}u_{a}rdr}{\int_{0}^{R}p_{a}^{2}u_{a}rdr}.\label{eq:Representation}
\end{equation}

\noindent Since $p_{a}$ is the minimizer in (\ref{eq:Representation}),
it solves the Euler-Lagrange equation
\begin{equation}
-p_{a}''u_{a}r-p_{a}'u_{a}'r-p_{a}'u_{a}=\mathfrak{e}\left(a(r)\right)p_{a}u_{a}r.\label{variational equation}
\end{equation}
Now we consider $\mathfrak{e}\left(\tilde{a}(r)\right)$. It follows
from the variational principle and (\ref{variational equation}) that
\[
\begin{aligned}\mathfrak{e}\left(\tilde{a}(r)\right) & \leq\frac{\int_{0}^{R}(p_{a}')^{2}u_{\tilde{a}}rdr}{\int_{0}^{R}p_{a}^{2}u_{\tilde{a}}rdr}=\frac{\int_{0}^{R}(-p_{a}''u_{a}r-p_{a}'u_{a}'r-p_{a}'u_{a})\frac{u_{\tilde{a}}}{u_{a}}p_{a}-p_{a}'p_{a}u_{a}r(\frac{u_{\tilde{a}}}{u_{a}})'dr}{\int_{0}^{R}p_{a}^{2}u_{\tilde{a}}rdr}\\
 & =\mathfrak{e}\left(a(r)\right)+\frac{2\int_{0}^{R}p_{a}'p_{a}(\tilde{a}-a)u_{\tilde{a}}rdr}{\int_{0}^{R}p_{a}^{2}u_{\tilde{a}}rdr}.
\end{aligned}
\]
Note that $p_{a}^{\prime}<0$ by Hopf's Lemma. 
\end{proof}
Proposition \ref{prop:rearrangement}, Lemma \ref{lem:quantitative comparision}
and the observation in (\ref{eq:variational}) allow us to conclude
with the following corollary. 
\begin{cor}
\label{cor:goal of part 1}Now let $f$ be a principal eigenfunction
corresponding to $\lambda(B,\Omega)$ and $q\left(\left|x\right|\right):=\left|f\right|^{*}\left(x\right)$.
Let $a(r)$ be as in Proposition \ref{prop:rearrangement} above,
and let $q_{a}$ be a normalized minimizer for the energy $\mathfrak{e}\left(a(r)\right)$
in (\ref{eq:disk energy}).  Then there is a universal constant $c>0$,
independent of $B$ and $\Omega$, such that 
\begin{equation}
\lambda(B,\Omega)\geq\lambda(B,D_{R})+c\int_{0}^{R}\left(q'(r)+a(r)q(r)\right)^{2}\mathcal{A}^{2}\left(\left\{ \left|f\right|>q(r)\right\} \right)rdr,\label{eq:rearrange error}
\end{equation}
and 
\begin{equation}
\lambda(B,\Omega)\geq\lambda(B,D_{R})+cB\frac{\int_{0}^{R}p_{a}\left|p_{a}^{\prime}\right|e^{-\frac{Br^{2}}{2}}\mathcal{A}^{2}\left(\left\{ \left|f\right|>q(r)\right\} \right)r^{2}dr}{\int_{0}^{R}p_{a}^{2}e^{-\frac{Br^{2}}{2}}rdr},\label{eq: compare error}
\end{equation}
where $p_{a}$ is as given in Lemma \ref{lem:quantitative comparision}
above. 
\end{cor}

\noindent Corollary \ref{cor:goal of part 1} implies $\lambda(B,\Omega)\geq\lambda(B,D_{R})$.
Furthermore if $\lambda(B,\Omega)=\lambda(B,D_{R})$, then either
(\ref{eq:rearrange error}) or (\ref{eq: compare error}) can be used
to deduce that almost all of the level sets of $\left|f\right|$ are
disks; and since $f$ is an analytic function, this implies $\Omega$
is a disk. 

The first bound, given in (\ref{eq:rearrange error}), is established
with \textit{our quantitative version of the rearrangement inequality}
and with Erdős' comparison lemma. In the absence of a magnetic field,
i.e. $B=0$, this bound reduces to the usual estimate used in all
the proofs of the quantitative Faber-Krahn inequality, e.g., \cite{Bhattacharya},
\cite{key-1} and \cite{key-4}. 

Our second bound, given in (\ref{eq: compare error}), is established
with Erdős' rearrangement inequality, \textit{our quantitative version
of the comparison lemma} and our estimate in (\ref{eq:new  potential}),
which follows from the quantitative isoperimetric inequality. This
bound, on the other hand, has no such analog in the absence of a magnetic
field. 

\part{\label{part:2}The Quantitative Version}

\noindent Here we prove Theorem \ref{thm:Main Result} from Corollary
\ref{cor:goal of part 1} by extracting the asymmetry of the whole
domain from the asymmetry of the level sets in (\ref{eq:rearrange error})
and (\ref{eq: compare error}). Let
\begin{equation}
\left|\left\{ q\left(\left|x\right|\right)>s\right\} \right|=\left|\Omega\right|\left(1-\frac{1}{2}\mathcal{A}\left(\Omega\right)\right).\label{eq: choice of s}
\end{equation}
Following Hansen and Nadirashvili \cite{key-4} we split the proof
into two cases, depending on whether $s$ is small or large. Lemma
\ref{lem:property of asymmetry} in the appendix will be useful. 

\section{\label{sec:The-First-Case:}The First Case: $s\lesssim e^{-BR^{2}}\mathcal{A}\left(\Omega\right)$}

\noindent We assume
\begin{equation}
s\leq\frac{1}{8}\left|\Omega\right|^{-\frac{1}{2}}e^{-\frac{BR^{2}}{4}}\mathcal{A}\left(\Omega\right).\label{eq: case 1}
\end{equation}

\noindent We use the representation in (\ref{eq:Representation}),
which allows us to adapt the usual strategy for dealing with the Dirichlet
Laplacian; and when $B=0$, the argument reduces to Hansen and Nadirashvili's
proof of their quantitative Faber-Krahn inequality \cite{key-4}. 

We write $E(B,\Omega):=\lambda(B,\Omega)-B$. Let $p:=qu_{a}^{-\frac{1}{2}}$
with $q,a$ as in Corollary \ref{cor:goal of part 1} and $u_{a}$
as in (\ref{eq:definition of u}), and let $\tilde{p}(r):=p(r)-se^{\int_{0}^{q^{-1}(s)}a(\tau)d\tau}$.
Since $\tilde{p}^{\prime}=p^{\prime}$, it follows from the rearrangement
inequality that 
\[
E(B,\Omega)\geq2\pi\int_{0}^{R}\left(q'+aq\right)^{2}rdr=2\pi\int_{0}^{R}\left(\tilde{p}^{\prime}\right)^{2}u_{a}rdr\geq2\pi\int_{0}^{q^{-1}(s)}\left(\tilde{p}^{\prime}\right)^{2}u_{a}rdr.
\]
Since $\tilde{p}$ vanishes at $q^{-1}(s)$, it is admissible in the
variational problem in (\ref{eq:Representation}) but on the disk
$\left\{ q>s\right\} $, and
\[
\frac{E(B,\Omega)}{2\pi\int_{0}^{q^{-1}(s)}\tilde{p}^{2}u_{a}rdr}\geq\inf_{p\in H_{0}^{1,\text{rad}}\left(\left\{ q>s\right\} \right)}\frac{\int_{0}^{q^{-1}(s)}\left(p^{\prime}\right)^{2}u_{a}rdr}{\int_{0}^{q^{-1}(s)}p^{2}u_{a}rdr}\geq E\left(B,\left\{ q>s\right\} \right),
\]
where the last inequality follows from the comparison lemma and the
observation in (\ref{eq:variational}). Using the scaling property
in Remark \ref{rem:scaling} we further estimate 
\begin{equation}
\frac{E(B,\Omega)}{2\pi\int_{0}^{q^{-1}(s)}\tilde{p}^{2}u_{a}rdr}\geq\frac{\left|\Omega\right|}{\left|\left\{ q>s\right\} \right|}E\left(B\frac{\left|\left\{ q>s\right\} \right|}{\left|\Omega\right|},D_{R}\right)\geq\frac{\left|\Omega\right|}{\left|\left\{ q>s\right\} \right|}E(B,D_{R}),\label{eq: first step}
\end{equation}
where the last inequality follows from Lemma \ref{lemma.disk} in
the appendix and again the comparison lemma. Finally, we estimate
the denominator 
\[
2\pi\int_{0}^{q^{-1}(s)}\tilde{p}^{2}u_{a}rdr=1-2\pi\int_{q^{-1}(s)}^{R}q^{2}rdr\ \ \ \ \ \ \ \ \ \ \ \ \ \ \ \ \ \ \ \ \ \ \ \ \ \ \ \ \ \ \ \ \ \ \ \ \ \ \ \ \ \ \ \ \ \ \ \ \ \ \ \ \ \ \ 
\]
\[
\ \ \ \ \ \ \ \ \ \ \ \ \ \ \ \ \ \ \ \ \ \ \ \ \ \ +2\pi\int_{0}^{q^{-1}(s)}\left((se^{\int_{0}^{q^{-1}(s)}a(\tau)d\tau})^{2}-2pse^{\int_{0}^{q^{-1}(s)}a(\tau)d\tau}\right)u_{a}rdr
\]
\[
\ \ \ \geq1-s^{2}\vert\{q<s\}\vert+s^{2}\vert\{q>s\}\vert-2se^{\frac{BR^{2}}{4}}\vert\Omega\vert^{\frac{1}{2}}
\]
\[
\geq1-2se^{\frac{BR^{2}}{4}}\vert\Omega\vert^{\frac{1}{2}}.\ \ \ \ \ \ \ \ \ \ \ \ \ \ \ \ \ \ \ \ \ \ \ \ \ \ \ \ \ \ \ \ 
\]
At the penultimate inequality we used that $e^{\int_{0}^{q^{-1}(s)}a(\tau)d\tau}\leq e^{\frac{BR^{2}}{4}}$
and that 
\[
2\pi\int_{0}^{q^{-1(s)}}pu_{a}rdr\leq2\pi\int_{0}^{R}qrdr\leq2\pi\vert\Omega\vert^{\frac{1}{2}}\int_{0}^{R}q^{2}rdr=\vert\Omega\vert^{\frac{1}{2}}.
\]
Combining the above estimate with (\ref{eq: first step}), we have
\[
E(B,\Omega)\geq E(B,D_{R})\frac{\vert\Omega\vert(1-2se^{\frac{BR^{2}}{4}}\vert\Omega\vert^{\frac{1}{2}})}{\vert\{q>s\}\vert},
\]
and the choice of $s$ in (\ref{eq: choice of s}) and our assumption
in (\ref{eq: case 1}) give us
\[
E(B,\Omega)\geq E(B,D_{R})\frac{1-\frac{1}{4}\mathcal{A}(\Omega)}{1-\frac{1}{2}\mathcal{A}(\Omega)}\geq E(B,D_{R})\Big(1+\frac{1}{4}\mathcal{A}(\Omega)\Big).
\]
Then using Lemma \ref{lem:bounds} in the appendix we find
\[
\lambda(B,\Omega)\geq\lambda(B,D_{R})\left(1+c\min(1,(BR^{2})^{-1}e^{-\frac{3}{4}BR^{2}})\mathcal{A}(\Omega)\right),
\]
which yields the desired estimates in (\ref{eq:quantitative}) and
(\ref{eq.B.small}). This concludes the proof of Theorem \ref{thm:Main Result}
in the first case. 

\section{The Second Case: $s\gtrsim e^{-BR^{2}}\mathcal{A}\left(\Omega\right)$}

\noindent We assume
\begin{equation}
s\geq\frac{1}{8}\left|\Omega\right|^{-\frac{1}{2}}e^{-\frac{BR^{2}}{4}}\mathcal{A}\left(\Omega\right).\label{eq: case 2}
\end{equation}

\noindent Now we have to treat weak and strong magnetic fields separately.
When $B$ is small, we only use the first bound, given in (\ref{eq:rearrange error}),
of Corollary \ref{cor:goal of part 1}. As $B$ increases, it becomes
necessary to also make use of our second bound in (\ref{eq: compare error}). 

\subsection{\label{subsec:Weak-Magnetic-Fields}Weak Magnetic Fields }

We consider $0\leq BR^{2}\leq\frac{1}{\pi}$ and prove the stability
estimate in (\ref{eq.B.small}); and when $B=0$, the argument reduces
to Hansen and Nadirashvili's proof of their quantitative Faber-Krahn
inequality \cite{key-4}. 

We work on the annulus $\left\{ q\left(\left|x\right|\right)\leq s\right\} $,
whose area is proportional to the asymmetry of the domain. From the
first bound, given in (\ref{eq:rearrange error}), of Corollary \ref{cor:goal of part 1},
the choice of $s$ in (\ref{eq: choice of s}), and Lemma \ref{lem:property of asymmetry}
we have
\[
\lambda(B,\Omega)-\lambda(B,D_{R})\ \ \ \ \ \ \ \ \ \ \ \ \ \ \ \ \ \ \ \ \ \ \ \ \ \ \ \ \ \ \ \ \ \ \ \ \ \ \ \ \ \ \ \ \ \ \ \ \ \ \ \ \ \ \ \ \ \ 
\]
\[
\geq c\int_{q^{-1}(s)}^{R}\left(q'(r)+a(r)q(r)\right)^{2}\mathcal{A}^{2}\left(\left\{ \left|f\right|>q(r)\right\} \right)rdr\ \ \ \ \ \ \ \ \ \ \ \ \ \ \ 
\]
\[
\geq cR^{2}\mathcal{A}^{2}\left(\Omega\right)\int_{q^{-1}(s)}^{R}\left(q'(r)+a(r)q(r)\right)^{2}r^{-1}dr\ \ \ \ \ \ \ \ \ \ \ \ \ \ \ \ \ \ \ \ \ \ 
\]
\[
\ \ \geq cR^{2}\mathcal{A}^{2}\left(\Omega\right)\left(\sqrt{\int_{q^{-1}(s)}^{R}q'(r)^{2}\,r^{-1}dr}-\sqrt{\int_{q^{-1}(s)}^{R}\left(\frac{B}{2}q\right)^{2}rdr}\right)^{2}
\]
\[
\geq cR^{2}\mathcal{A}^{2}\left(\Omega\right)\left(\frac{s}{\sqrt{\left|\left\{ q\left(\left|x\right|\right)\leq s\right\} \right|}}-\frac{B}{2}s\sqrt{\left|\left\{ q\left(\left|x\right|\right)\leq s\right\} \right|}\right)^{2}\ \ \ \ \ \ \ 
\]
\[
\geq cR^{-2}\mathcal{A}^{3}\left(\Omega\right)\left(2-B\left|\Omega\right|\right)^{2}\ \ \ \ \ \ \ \ \ \ \ \ \ \ \ \ \ \ \ \ \ \ \ \ \ \ \ \ \ \ \ \ \ \ \ \ \ \ \ \ \ \ \ \ \ \ 
\]
\[
\geq cR^{-2}\mathcal{A}^{3}\left(\Omega\right),\ \ \ \ \ \ \ \ \ \ \ \ \ \ \ \ \ \ \ \ \ \ \ \ \ \ \ \ \ \ \ \ \ \ \ \ \ \ \ \ \ \ \ \ \ \ \ \ \ \ \ \ \ \ \ \ \ \ \ \ \ 
\]
since $B\leq\frac{1}{\left|\Omega\right|}=\frac{1}{\pi R^{2}}$. At
the penultimate inequality we also used the assumption in (\ref{eq: case 2}).
Using Lemma \ref{lem:bounds}, we conclude $\lambda(B,\Omega)\geq\lambda(B,D_{R})(1+c\mathcal{A}(\Omega)^{3})$. 

\subsection{Strong Magnetic Fields}

We consider $BR^{2}>\frac{1}{\pi}$ and prove our stability estimate
in (\ref{eq:quantitative}); instead of integrating as above on $\left\{ q\left(\left|x\right|\right)\leq s\right\} $,
we choose to work closer to the boundary on a smaller annulus whose
area is now proportional to the \textit{spectral defici}t of the domain
\[
\mathcal{D}(B,\Omega):=\frac{\lambda(B,\Omega)}{\lambda(B,D_{R})}-1.
\]
We treat two cases, depending on whether $q$ is large or small near
the boundary: $q(R(1-\mathcal{D}\left(B,\Omega\right)^{\alpha}))>R^{-1}\mathcal{D}(B,\Omega)^{\beta}$
and $q(R(1-\mathcal{D}\left(B,\Omega\right)^{\alpha}))\leq R^{-1}\mathcal{D}(B,\Omega)^{\beta}$,
where $\alpha=\frac{1}{5}$ and $\beta=\frac{3}{10}$ are chosen to
optimize our result. For proving our estimate in (\ref{eq:quantitative}),
we can assume that the spectral deficit is very small
\begin{equation}
\mathcal{D}(B,\Omega)^{\alpha}<\min\left\{ \frac{1}{2BR^{2}},\frac{1}{2}\right\} .\label{eq:suffices}
\end{equation}

\subsubsection{\label{subsec:q is large} Suppose $q(R(1-\mathcal{D}\left(B,\Omega\right)^{\alpha}))>R^{-1}\mathcal{D}(B,\Omega)^{\beta}$.}

Then by continuity of $q$,
\begin{equation}
q(R(1-\mathcal{D}\left(B,\Omega\right)^{\tilde{\alpha}}))=R^{-1}\mathcal{D}(B,\Omega)^{\beta}\ \text{for some}\ \tilde{\alpha}>\alpha.\label{eq:tilde}
\end{equation}
If $q(R(1-\mathcal{D}(B,\Omega)^{\tilde{\alpha}}))\geq s$, our assumption
in (\ref{eq: case 2}) readily yields
\[
cR^{-1}e^{-\frac{BR^{2}}{4}}\mathcal{A}\left(\Omega\right)\leq s\leq q(R(1-\mathcal{D}\left(B,\Omega\right)^{\tilde{\alpha}}))=R^{-1}\mathcal{D}(B,\Omega)^{\beta},
\]
and therefore 
\begin{equation}
\mathcal{D}(B,\Omega)\ge ce^{-\frac{BR^{2}}{4\beta}}\mathcal{A}\left(\Omega\right)^{\frac{1}{\beta}}.\label{eq:I-1}
\end{equation}

If $q(R(1-\mathcal{D}(B,\Omega)^{\tilde{\alpha}}))<s$, then the weak-field
argument from Section \ref{subsec:Weak-Magnetic-Fields} applies mutatis
mutandis. From the first bound, given in (\ref{eq:rearrange error}),
of Corollary \ref{cor:goal of part 1}, the relation in (\ref{eq:tilde}),
and Lemma \ref{lem:property of asymmetry} we have
\begin{align*}
 & \lambda(B,D_{R})\mathcal{D}\left(B,\Omega\right)\\
 & \geq c\int_{R\left(1-\mathcal{D}(B,\Omega)^{\tilde{\alpha}}\right)}^{R}\left(q'(r)+a(r)q(r)\right)^{2}\mathcal{A}^{2}\left(\left\{ \left|f\right|>q(r)\right\} \right)rdr\\
 & \geq cR^{2}\mathcal{A}(\Omega)^{2}\left(\frac{q\left(R\left(1-\mathcal{D}(B,\Omega)^{\tilde{\alpha}}\right)\right)}{\sqrt{2R^{2}\mathcal{D}(B,\Omega)^{\tilde{\alpha}}}}-\frac{B}{2}q\left(R\left(1-\mathcal{D}(B,\Omega)^{\tilde{\alpha}}\right)\right)\sqrt{2R^{2}\mathcal{D}(B,\Omega)^{\tilde{\alpha}}}\right)^{2}\\
 & =cR^{-2}\mathcal{A}(\Omega)^{2}\mathcal{D}\left(B,\Omega\right)^{2\beta-\tilde{\alpha}}\left(1-BR^{2}\mathcal{D}(B,\Omega)^{\tilde{\alpha}}\right)^{2}.
\end{align*}
However, $\tilde{\alpha}$ depends on $B$ and $\Omega$. Fortunately
since $\tilde{\alpha}>\alpha$ and $\mathcal{D}\left(B,\Omega\right)<1$,
we have $\mathcal{D}\left(B,\Omega\right)^{\tilde{\alpha}}<\mathcal{D}\left(B,\Omega\right)^{\alpha}$;
this allows to replace $\mathcal{D}\left(B,\Omega\right)^{\tilde{\alpha}}$
in the above with $\mathcal{D}\left(B,\Omega\right)^{\alpha}$. Furthermore,
the bound in (\ref{eq:suffices}) offsets the large $BR^{2}$ in the
parenthetical expression, which thereby remains positive. Using Lemma
\ref{lem:bounds},
\[
\mathcal{D}\left(B,\Omega\right)\geq c\frac{\mathcal{A}(\Omega)^{2}}{R^{2}\lambda(B,D_{R})}\mathcal{D}\left(B,\Omega\right)^{2\beta-\alpha}\geq c\frac{\mathcal{A}(\Omega)^{2}}{1+BR^{2}}\mathcal{D}\left(B,\Omega\right)^{2\beta-\alpha},
\]
and therefore 
\begin{equation}
\mathcal{D}\left(B,\Omega\right)^{1-2\beta+\alpha}\ge c\frac{\mathcal{A}(\Omega)^{2}}{1+BR^{2}}.\label{eq:I-2}
\end{equation}
With our above choice of $\alpha$ and $\beta$, the inequalities
in (\ref{eq:I-1}) and (\ref{eq:I-2}) both yield the same desired
estimate in (\ref{eq:quantitative}). 

Thus far, we have only used the first bound, given in (\ref{eq:rearrange error}),
of Corollary \ref{cor:goal of part 1} which is based on the quantitative
version of the rearrangement inequality. 

\subsubsection{\label{subsec:q is small}Suppose $q(R(1-\mathcal{D}\left(B,\Omega\right)^{\alpha}))\protect\leq R^{-1}\mathcal{D}(B,\Omega)^{\beta}$. }

If $q(R(1-\mathcal{D}(B,\Omega)^{\alpha}))\geq s,\ \ $ again our
assumption in (\ref{eq: case 2}) readily yields
\[
cR^{-1}e^{-\frac{BR^{2}}{4}}\mathcal{A}\left(\Omega\right)\leq s\leq q(R(1-\mathcal{D}\left(B,\Omega\right)^{\alpha}))\leq R^{-1}\mathcal{D}(B,\Omega)^{\beta}
\]
and therefore, as above, 
\begin{equation}
\mathcal{D}(B,\Omega)\geq ce^{-\frac{BR^{2}}{4\beta}}\mathcal{A}\left(\Omega\right)^{\frac{1}{\beta}}.\label{eq:II-1}
\end{equation}

But when $q(R(1-\mathcal{D}(B,\Omega)^{\alpha}))<s$, the weak-field
argument from Section \ref{subsec:Weak-Magnetic-Fields} is no longer
useful: it requires a \textit{lower bound} on $q(R(1-\mathcal{D}(B,\Omega)^{\alpha}))$,
as above in Section \ref{subsec:q is large}, to be effective. That
argument, however, is based wholly on the first bound, given in (\ref{eq:rearrange error}),
of Corollary \ref{cor:goal of part 1}. 

Now we instead turn to our second bound, given in (\ref{eq: compare error}),
which is based on our quantitative version of the comparison lemma.
Here there is hope: it is possible to bound the remainder term in
(\ref{eq: compare error}) from below \textit{independently of} $q$. 
\begin{lem}
\label{lem:Comparison lower bound}Let $p_{a}$ be as in Corollary
\ref{cor:goal of part 1}. Then there exists a universal constant
$c>0$, independent of $B$ and $\Omega$, such that for any $\,0<\varepsilon<\frac{1}{2}$
\[
\frac{\int_{R\left(1-\varepsilon\right)}^{R}\,p_{a}\left|p_{a}^{\prime}\right|e^{-\frac{Br^{2}}{2}}\mathcal{A}^{2}\left(\left\{ \left|f\right|>q(r)\right\} \right)r^{2}dr}{\int_{0}^{R}p_{a}^{2}e^{-\frac{Br^{2}}{2}}rdr}\geq ce^{-\frac{BR^{2}}{2}}\mathcal{M}_{\varepsilon}\varepsilon^{2},
\]
where $\mathcal{M}_{\varepsilon}:=\inf\left\{ \mathcal{A}^{2}\left(\left\{ \left|f\right|>q(r)\right\} \right):\,R\left(1-\varepsilon\right)<r<R\right\} $. 
\end{lem}

\begin{proof}
Since $p_{a}^{\prime}<0$, 
\begin{align*}
 & \int_{R\left(1-\varepsilon\right)}^{R}\,p_{a}\left|p_{a}^{\prime}\right|e^{-\frac{Br^{2}}{2}}\mathcal{A}^{2}\left(\left\{ \left|f\right|>q(r)\right\} \right)r^{2}dr\\
 & \geq cM_{\varepsilon}R^{2}e^{-\frac{BR^{2}}{2}}\int_{R(1-\varepsilon)}^{R}-p_{a}(r)p_{a}'(r)dr=cM_{\varepsilon}R^{2}e^{-\frac{BR^{2}}{2}}p_{a}(R(1-\varepsilon))^{2}.
\end{align*}

\noindent Furthermore, 
\[
p_{a}(R(1-\varepsilon))=\int_{R(1-\varepsilon)}^{R}-p_{a}^{\prime}(r)dr\geq\frac{1}{R}\int_{R(1-\varepsilon)}^{R}-p_{a}^{\prime}(r)rdr\geq\frac{\varepsilon}{R}\int_{0}^{R}-p_{a}^{\prime}(r)rdr,
\]
where in the last inequality we used that $r\mapsto-p_{a}^{\prime}(r)r$
is increasing (see (\ref{variational equation})). The lemma follows
from the Sobolev inequality

\noindent $\ \ \ \ \ \ \ \ \ \ \ \ \ \int_{0}^{R}-p_{a}^{\prime}(r)rdr\geq c\left(\int_{0}^{R}p_{a}^{2}(r)rdr\right)^{\frac{1}{2}}\geq c\left(\int_{0}^{R}p_{a}^{2}(r)\,e^{-\frac{Br^{2}}{2}}rdr\right)^{\frac{1}{2}}.$
\end{proof}
Before proceeding with our argument, we remark that Lemma \ref{lem:Comparison lower bound}
would not have been useful for dealing with the previous situation
in Section \ref{subsec:q is large}.

If $q(R(1-\mathcal{D}(B,\Omega)^{\alpha}))<s$, then we use the above
lemma with $\varepsilon=\mathcal{D}\left(B,\Omega\right)^{\alpha}$.
From our second bound, given in (\ref{eq: compare error}), of Corollary
\ref{cor:goal of part 1}, Lemma \ref{lem:Comparison lower bound},
and Lemma \ref{lem:property of asymmetry} we have
\[
\lambda(B,D_{R})\mathcal{D}\left(B,\Omega\right)\geq cBe^{-\frac{BR^{2}}{2}}\mathcal{A}\left(\Omega\right)^{2}\mathcal{D}\left(B,\Omega\right)^{2\alpha}.
\]
Again using Lemma \ref{lem:bounds} and now that $BR^{2}>\frac{1}{\pi}$,
\[
\mathcal{D}\left(B,\Omega\right)\geq c\frac{e^{-\frac{BR^{2}}{2}}}{1+\left(BR^{2}\right)^{-1}}\mathcal{A}\left(\Omega\right)^{2}\mathcal{D}\left(B,\Omega\right)^{2\alpha}\geq ce^{-\frac{BR^{2}}{2}}\mathcal{A}\left(\Omega\right)^{2}\mathcal{D}\left(B,\Omega\right)^{2\alpha},
\]
and therefore 
\begin{equation}
\mathcal{D}\left(B,\Omega\right)^{1-2\alpha}\geq ce^{-\frac{BR^{2}}{2}}\mathcal{A}\left(\Omega\right)^{2}.\label{eq:II-2}
\end{equation}
With our above choice of $\alpha$ and $\beta$, the inequalities
in (\ref{eq:II-1}) and (\ref{eq:II-2}) both yield the same desired
estimate in (\ref{eq:quantitative}). This concludes the proof of
Theorem \ref{thm:Main Result}. 

\section*{Acknowledgements}

\noindent We are most grateful to Søren Fournais for encouraging our
collaboration. R.G. first suggested this problem in October 2017 to
Michael Loss, whom he thanks for the initial encouragement. This paper
is based on work partially supported by the Independent Research Fund
Denmark via the project grant “Mathematics of the dilute Bose gas”
No. 0135-00166B (L.J. \& L.M.). 

\appendix

\section{The Magnetic Laplacian on the Disk}

\noindent It follows from Erdős' rearrangement inequality and comparison
lemma, and from the observation in (\ref{eq:variational}) that the
principal eigenfunction of the magnetic Laplacian on the disk is radially
symmetric.
\begin{thm}
\label{thm:Symmetry}As above, let $D_{R}$ be a disk of radius $R$
centered at the origin. Then
\[
\lambda(B,D_{R})=\inf_{q\in H_{0}^{1,\text{rad}}\left(D_{R}\right)}\frac{\int_{D_{R}}\left|\left(-i\nabla-\alpha\right)q(\left|x\right|)\right|^{2}dx}{\int_{D_{R}}q\left(\left|x\right|\right)^{2}dx},
\]
where $H_{0}^{1,\text{rad}}(D_{R}):=\left\{ q:[0,R]\rightarrow\mathbb{R}\ \text{such that}\ x\mapsto q(\vert x\vert)\ \text{belongs to}\ H_{0}^{1}(D_{R})\right\} .$ 
\end{thm}

\noindent Thus we write $\lambda(B,D_{R})$ more conveniently in terms
of polar coordinates. 
\begin{lem}
\label{lemma.disk} Let $H_{0}^{1,\text{rad}}(D_{R})$ be as in Theorem
\ref{thm:Symmetry}. Then
\[
\lambda(B,D_{R})=B+\inf_{q\in H_{0}^{1,\text{rad}}(D_{R})}\frac{2\pi\int_{0}^{R}(q'(r)+\frac{Br}{2}q(r))^{2}rdr}{2\pi\int_{0}^{R}q(r)^{2}rdr}=:B+\mathfrak{e}\left(Br/2\right).
\]
\end{lem}

\begin{proof}
First we consider a broader class of vector potentials $\tilde{\alpha}(x):=\frac{a\left(\left|x\right|\right)}{\left|x\right|}\left(-x_{2},x_{1}\right)$
on the disk, with $a\left(\left|x\right|\right)$ bounded. These correspond
to radially symmetric but possibly inhomogeneous magnetic fields that
show up in the rearrangement inequality. Written in polar coordinates,
$\tilde{\alpha}(r,\theta)=a(r)\left(-\sin\theta,\cos\theta\right)$
and for $f\in H_{0}^{1}(D_{R})$
\[
\int_{D_{R}}\left|\left(-i\nabla-\tilde{\alpha}\right)f\right|^{2}dx=\int_{0}^{R}\int_{0}^{2\pi}\left(\vert\partial_{r}f\vert^{2}+\vert\frac{i}{r}\partial_{\theta}f+af\vert^{2}\right)rd\theta dr.
\]
Thus for any $q\in H_{0}^{1,\text{rad}}(D_{R})$, 
\begin{align*}
\int_{D_{R}}\left|\left(-i\nabla-\tilde{\alpha}\right)q(\left|x\right|)\right|^{2}dx & =2\pi\int_{0}^{R}\left((q'(r)^{2}+\left(a(r)q(r)\right)^{2}\right)rdr\\
 & =2\pi\int_{0}^{R}\left(q'(r)+a(r)q(r)\right)^{2}rdr-2\pi\int_{0}^{R}\left(q^{2}\right)^{\prime}a(r)rdr,
\end{align*}
and after integrating by parts 
\begin{align*}
\int_{D_{R}}\left|\left(-i\nabla-\tilde{\alpha}\right)q(\left|x\right|)\right|^{2}dx & =2\pi\int_{0}^{R}\left(q'(r)+a(r)q(r)\right)^{2}rdr+2\pi\int_{0}^{R}q^{2}\left(a(r)r\right)^{\prime}dr\\
 & =2\pi\int_{0}^{R}\left(q'(r)+a(r)q(r)\right)^{2}rdr+\int_{D_{R}}\text{rot}\left(\tilde{\alpha}\right)q(\left|x\right|)^{2}dx.
\end{align*}

Returning to the original potential $\alpha=\frac{B}{2}\left(-x_{2},x_{1}\right)$,
the lemma follows from Theorem \ref{thm:Symmetry}, the above calculation
and that $\text{rot}\left(\alpha\right)=B$. 
\end{proof}
\noindent Moreover, Erdős proved the following estimates. See Proposition
A.1 in \cite{key-79}. 
\begin{lem}
\label{lem:bounds}There are universal constants $C_{1},C_{2}$ such
that
\[
B+\frac{C_{1}}{R^{2}}e^{-\frac{3}{4}BR^{2}}\leq\lambda\left(B,D_{R}\right)\leq B+C_{2}B\left(\frac{1}{BR^{2}}+BR^{2}\right)e^{-\frac{1}{8}BR^{2}}.
\]
\end{lem}

\noindent Improving these estimates is an ongoing area of research
\cite{key-9},\cite{key-10},\cite{key-11} \& ref. therein. In the
absence of a magnetic field, $\lambda(0,D_{R})=j_{0,1}^{2}R^{-2}$
where $j_{0,1}\approx2.4048$ is the first zero of the Bessel function
of order zero.

\section{\label{sec:Asymmetry of Large Subsets}Asymmetry of Large Subsets}

\noindent If a subset is large enough, its asymmetry is comparable
to the asymmetry of the whole domain \cite{key-3},\cite{key-4}. 
\begin{lem}
\label{lem:property of asymmetry}Let $U\subseteq\Omega$ with $\left|U\right|=\pi r^{2}$
and $\left|\Omega\right|=\pi R^{2}$. If $\left|U\right|\geq\left|\Omega\right|\left(1-\frac{1}{2}\mathcal{A}\left(\Omega\right)\right)$,
then $r\mathcal{A}\left(U\right)\geq\frac{1}{2}R\mathcal{A}\left(\Omega\right)$. 
\end{lem}

\begin{proof}
First we consider the interior asymmetry. From our assumption on the
area of $U$, we have $\left|U\right|\geq\left|\Omega\right|\left(1-\frac{1}{2}\mathcal{A}_{I}\left(\Omega\right)\right)^{2}$
and thus $r\geq R\left(1-\frac{1}{2}\mathcal{A}_{I}\left(\Omega\right)\right)$.
We then deduce that $r-\rho_{-}(U)\geq r-\rho_{-}(\Omega)\geq\frac{1}{2}\left(R-\rho_{-}(\Omega)\right)$,
which yields the lemma. 

Now we turn to the Fraenkel asymmetry. Let $D_{U}$ and $D_{\Omega}$
denote two concentric balls such that $\left|D_{U}\right|=\left|U\right|$
and $\left|D_{\Omega}\right|=\left|\Omega\right|$. Then, $\left|D_{\Omega}\triangle\Omega\right|\leq\left|D_{U}\triangle U\right|+2\left(\left|\Omega\right|-\left|U\right|\right).$
Using this inequality and our assumption on the area of $U$, we deduce
\[
\frac{\vert D_{U}\Delta U\vert}{2\vert U\vert}\geq\frac{\vert D_{\Omega}\Delta\Omega\vert}{2\vert U\vert}-\frac{\vert\Omega\vert-\vert U\vert}{\vert U\vert}\geq\frac{1}{2}\mathcal{A}_{F}(\Omega)\frac{\vert\Omega\vert}{\vert U\vert}\geq\frac{1}{2}\frac{R}{r}\mathcal{A}_{F}(\Omega).
\]
Taking the infimum over all translations of $D_{U}$ concludes the
proof. 
\end{proof}

\end{document}